\documentclass[11pt, reqno]{amsart}
\usepackage{enumitem}

\numberwithin{equation}{section}

\newtheorem*{claim}{Claim}
\newtheorem{theorem}{Theorem}[section]

\newtheorem{lemma}[theorem]{Lemma}

\newtheorem{remark}[theorem]{Remark}

\def\al{\aligned}
\def\eal{\endaligned}
\def\be{\begin{equation}}
\def\ee{\end{equation}}
\def\lab{\label}
\def\a{\alpha}

\def\e{\epsilon}

\def\M{{\bf M}}

\def\al{\aligned}

\def\pa{\partial}
\def\d{\nabla}

\numberwithin{equation}{section}

\begin{document}

\title[]{Li-Yau gradient bound for collapsing manifolds under integral curvature condition}
\author{Qi S. Zhang and Meng Zhu}
\address{Department of Mathematics, University of California, Riverside, CA 92521, USA}
\email{qizhang@math.ucr.edu}
\address{Department of Mathematics, East China Normal University, Shanghai 200241, China
-and- Department of Mathematics, University of California, Riverside, CA 92521, USA}
\email{mzhu@math.ucr.edu}
\date{}

\begin{abstract}
Let $(\M^n, g_{ij})$ be a complete Riemammnian manifold. For some constants $p,\ r>0$, define $\displaystyle k(p,r)=\sup_{x\in M}r^2\left(\oint_{B(x,r)}|Ric^-|^p dV\right)^{1/p}$, where $Ric^-$ denotes the negative part of the Ricci curvature tensor. We prove that for any $p>\frac{n}{2}$, when $k(p,1)$ is small enough, certain Li-Yau type gradient bound holds for the positive solutions of the heat equation on geodesic balls $B(O,r)$ in $\M$ with $0<r\leq 1$. Here the assumption that $k(p,1)$ being small allows the situation where the manifolds is collapsing. Recall that in \cite{ZZ}, certain Li-Yau gradient bounds was also obtained by the authors, assuming that $|Ric^-|\in L^p(\M)$ and the manifold is noncollaped. Therefore, to some extent, the results in this paper and in \cite{ZZ} complete the picture of Li-Yau gradient bound for the heat equation on manifolds with $|Ric^-|$ being $L^p$ integrable, modulo sharpness of constants.

\end{abstract}
\maketitle

\section{Introduction}

Let $(\M^n, g_{ij})$ be a complete Riemmanian manifold. In \cite{LY}, Li and Yau proved the celebrated Li-Yau gradient bound for positive solutions of the heat equation on $\M$ when the Ricci curvature of $\M$ is bounded from below. It states that if $Ric\geq -K$ for some constant $K\geq 0$, then for any positive solution $u$ of the heat equation $\pa _t u=\Delta u$, one has
\begin{equation}\label{Li-Yau>1}
\frac{|\nabla u|^2}{u^2} - \alpha\frac{u_t}{u} \leq \frac{n\alpha^2K}{2(\alpha-1)}+\frac{n\alpha^2}{2t},\quad \forall \alpha>1.
\end{equation}
Especially when $Ric\geq 0$, one has the optimal Li-Yau bound
\begin{equation}\label{Li-Yau=1}
\frac{|\nabla u|^2}{u^2} - \frac{u_t}{u} \leq \frac{n}{2t}.
\end{equation}

Li-Yau type gradient bounds for parabolic equations are essential tools for studying topological and geometrical properties of manifolds. For instance, the classical parabolic Harnack inequality, optimal Gaussian estimates of the heat kernel, estimates of eigenvalues of the Laplace operator, estimates of the Green's function, and even Laplacian comparison theorem can be deduced from \eqref{Li-Yau=1}.

While the coefficients in \eqref{Li-Yau=1} are sharp for the case where $Ric\geq 0$, for the case where the Ricci curvature bounded from below, many efforts have been made to improve \eqref{Li-Yau>1} in the past several decades. The readers  may refer to \cite{Ha}, \cite{CTZ}, \cite{Dav}, \cite{GM}, \cite{LX}, \cite{QZZ}, \cite{Wan}, \cite{WanJ} and the latest \cite{BBG} and the references therein for more information.

On the other hand, generalizations of Li-Yau gradient bounds have also been studied by many mathematicians. Hamilton \cite{Ha} discovered a matrix Li-Yau type bound for the heat equation. Certain matrix Li-Yau bound under weaker conditions was subsequently obtained by Cao-Ni \cite{CaNi} on K\"ahler manifolds. Moreover, Li-Yau type bounds were also proved for weighted manifolds with Bakry-\'Emery Ricci curvature being bounded from below, or more generally for metric measure spaces $(X, d, \mu)$ satisfying $RCD^*(K,N)$ condition (see e.g. \cite{BL} and \cite{ZhZx}).

However, to authors' knowledge, in all the known results on the standard Li-Yau bounds on Riemmanian manifolds, the lower boundedness of the Ricci curvature is a necessary assumption. In \cite{ZZ}, the authors initiated an effort to derive Li-Yau bounds for relaxed Ricci curvature condition. More precisely, certain Li-Yau gradient bound for positive solution of the heat equation on compact manifolds was proved by assuming that either $|Ric^-|\in L^p(\M)$ for some $p>n/2$ and the manifold is noncollapsed, or certain Kato type of norm of $|Ric^-|$ is finite and the heat kernel has a Gaussian upper bound. Here $Ric^-$ is the negative part of the Ricci curvature tensor. Both assumptions allow the lower bound of the Ricci curvature to tend to $-\infty$.

In this paper, we will extend the result in \cite{ZZ} for the case where $|Ric^-|\in L^p(\M)$ by removing the noncollapsing assumption with the assistance of the Sobolev inequality proved by Dai-Wei-Z.L. Zhang \cite{DWZ} most recently. Our main theorem is

\begin{theorem}\label{main thm}
Let $(\M^n, g_{ij})$ be a complete Riemannian manifold and $u$ a positive solution of the heat equation on $\M$, i.e.,
\be\label{HE}
(\Delta - \pa_t)u=0.
\ee
For any $p>\frac{n}{2}$, there exists a constant $\kappa=\kappa(n,p)$ such that the following holds. If $k(p,1)\leq \kappa$, then for any point $O\in \M$ and constant $0<\alpha<1$, we have
\be\label{Li-Yau}
\a \underline{J}\frac{|\d u|^2}{u^2}-\frac{\pa_t u}{u}\leq \frac{n}{\a(2 - \delta)\underline{J}}\frac{1}{t}+\frac{C}{\a(2 - \delta)\underline{J}}\left[\frac{1}{\a(2-\delta)\underline{J}(1-\a)}+1\right],
\ee
in $B(O,\frac{1}{2})\times(0,\infty)$, where
\[
\underline{J}=\underline{J}(t)=2^{-\frac{1}{a-1}}\exp\left\{-2C\kappa\left(1+[2C(a-1)\kappa]^{\frac{n}{2p-n}}\right)t\right\},
\] $\delta=\frac{2(1-\a)^2}{n+(1-\a)^2}$, $a= \frac{5[n+(1-\alpha)^2]}{2(1-\alpha)^2}$ and $C=C(n,p)$ is a constant depending on $n$ and $p$.
\end{theorem}

Comparing to Theorem 1.1 in \cite{ZZ}, the assumption $k(p,1)\leq \kappa$ in the above theorem includes the possibility that the manifold is collapsing, while the Li-Yau bound above is a local estimate instead of a global one, since the assumption is only made locally.  The theorem clearly implies a local Harnack inequality which also implies a Gaussian lower bound for the heat kernel.

The method of the proof is similar to the proof of Theorem 1.1 in \cite{ZZ}. Namely, we compute the evolution of the quantity $\a J\frac{|\d u|^2}{u^2}-\frac{\pa_t u}{u}$ with $J=J(x,t)$ being a smooth function. To be able to proceed with the maximum principle argument, we need to solve certain nonlinear parabolic equation involving $J$ and derive upper and lower bounds of $J$, and this is where the Gaussian upper bound of the heat kernel and the volume doubling property come into play. The difference is that here we need to consider a boundary value problem for $J$ instead of the Cauchy problem in the proof of Theorem 1.1 in \cite{ZZ}.

Finally, let us note that \eqref{Li-Yau} is a scaling invariant inequality (see Remark \ref{rmk} below).
\smallskip

\section{Proof of the main theorem}
Let $(\M^n, g_{ij})$ be a complete Riemannian manifold. In the following, we use $B(x,r)$ and $|B(x,r)|$ to denote the geodesic ball with radius $r$ in $\M$ centered at $x$ and its volume, respectively. Also, the notation $\oint_{B(x,r)}$ represents the average integral $\frac{1}{|B(x,r)|}\int_{B(x,r)}$ on $B(x,r)$.

For $p, r>0$, following \cite{DWZ}, define
\[
k(x,p,r)=r^2\left(\oint_{B(x,r)}|Ric^-|^p\right)^{1/p},\qquad\  k(p,r)=\sup_{x\in \M}k(x,p,r).
\]

The main tools used in the proof of Theorem \ref{main thm} are volume doubling property and Gaussian upper bound of the heat kernel. Firstly, the volume doubling property was proved in \cite{PeWe}, i.e.,

\begin{lemma}[Petersen-Wei \cite{PeWe} Corollary 1.2]\label{lem volume doubling}
For any $p>n/2$ there is a $\kappa=\kappa(n,p)$ such that if $k(p,1)\leq \kappa$, then for all $x\in \M$ and $0<r_1\leq r_2\leq 1$, we have
\be\label{eq volume doubling}
\frac{|B(x,r_2)|}{r_2^n}\leq 2\frac{|B(x,r_1)|}{r_1^n}.
\ee
\end{lemma}

To get a Gaussian upper bound of the heat kernel, one also needs some estimate of the Sobolev constant besides the volume doubling property, which was obtained as follows in \cite{DWZ}.

\begin{lemma}[Dai-Wei-Zhang \cite{DWZ} Corollary 1.5]\label{lem Sobolev}
For any $p>n/2$ there is a $\kappa=\kappa(n,p)$ such that if $k(p,1)\leq \kappa$, then for all $x\in \M$, $0<r\leq 1$ and $f\in C_0^{\infty}(M)$ we have
\be\label{eq Sobolev}
\left(\oint_{B(x,r)}|f|^{\frac{n}{n-1}}dV\right)^{\frac{n-1}{n}}\leq C(n)r\oint_{B(x,r)}|\d f|dV.
\ee

\end{lemma}

Denote by $G(x,t;y,0)$ the heat kernel of \eqref{HE}. With the Sobolev inequality \eqref{eq Sobolev} and volume doubling property \eqref{eq volume doubling}, it is well known that one can derive the following Gaussian upper bound for $G$ (see e.g. \cite{Sa}).

\begin{lemma}
For any $p>n/2$ there is a $\kappa=\kappa(n,p)$ such that if $k(p,1)\leq \kappa$, then for some constants $C_i=C_i(n,p)$, $i=1,2$, we have
\be\label{Gaussian upper bound}
G(x,t;y,0)\leq \frac{C_1}{|B(x,\sqrt{t})|^{\frac{1}{2}}|B(y,\sqrt{t})|^{\frac{1}{2}}}e^{-\frac{d^2(x,y)}{C_2t}}.
\ee
\end{lemma}

For the maximum principle to work locally, we also need to the following cut-off function.

\begin{lemma}[Dai-Wei-Zhang \cite{DWZ} Lemma 5.3]\label{cutoff}
Let $(\M^n,g_{ij})$ be a complete Riemannian manifold. Then for any $p>\frac{n}{2}$, there exist constants $\kappa=\kappa(n,p)$ and $C=C(n,p)$ such that if $k(p,1)\leq \kappa$, then for any geodesic ball $B(x,r)$ and $0<r\leq 1$ there exists $\phi\in C_0^{\infty}(B(x,r))$ satisfying $0\leq \phi \leq 1$, $\phi\equiv 1$ in $B(x,\frac{r}{2})$, and $|\d \phi|^2+|\Delta \phi|\leq \frac{C}{r^2}$.
\end{lemma}

Now we are ready to prove the main theorem.

\begin{proof}[Proof of Theorem \ref{main thm}]
Let $J=J(x, t)$ be a smooth positive  function and
\[
Q(x,t)=\a J\frac{|\d u|^2}{u^2}-\frac{\pa_t u}{u}.
\]
According to the computations in the proof of Theorem 1.1 in \cite{ZZ}, we have
\be
\lab{ineqtQ}
\al
&(\Delta  - \pa_t) ( t Q )
+ 2 \frac{\d u}{u} \d (t Q)  \\
\ge&  \a t \frac{2-\delta}{n} J \left( | \d  f |^2 - \pa_t f \right)^2  +
\a \left[\Delta J - 2 V  J
-
5 \delta^{-1} \frac{| \d J|^2}{J}
 -\pa_t J \right]  t |\d f|^2\\
&  - \delta \a  t J |\d f|^4 -Q,
\eal
\ee where $V=| Ric^-|$ and $f=\ln u$.

For a fixed point $O\in \M$ and any given parameter $\delta>0$ such that $5\delta^{-1}>1$, we make the following
\begin{claim}
\lab{claim}
there exists a $\kappa=\kappa(n,p)$ such that when $k(p,1)\leq \kappa$, for any $0<r\leq 1$, the problem
\be
\lab{eqforJ}
\begin{cases}
\Delta J - 2 V  J
-
5 \delta^{-1}\frac{| \d J|^2}{J}
 -\pa_t J =0, \quad \text{on} \quad {B(O,r)} \times (0, \infty);\\
J(\cdot, 0) = 1,\ on\ B(O,r)\\
J(\cdot, t) = 1,\ on\ \pa B(O,r)
\end{cases}
\ee
has a unique solution for $t\in[0,\infty)$, which satisfies
\be
\underline{J}_r(t)\leq J(x,t)\leq 1,
\ee
where
\be\label{J_r}
\underline{J}_r(t)=2^{-\frac{1}{a-1}}\exp\left\{-2C\kappa r^{-2}\left(1+[2C(a-1)\kappa]^{\frac{n}{2p-n}}\right)t\right\}
\ee
for some constant $C=C(n,p)$ and $a=5\delta^{-1}$.
\end{claim}

In the following steps, we will prove the claim.\\

\hspace{-.45cm}{\it step 1.   Conversion into an integral equation.}

Let $a=5\delta^{-1}$, and
\be
\lab{defw}
w = J^{-(a-1)}.
\ee
It is straightforward to check that $w$ satisfies
\be
\lab{eqforw}
\begin{cases}
\Delta w -\pa_t w +2(a-1) V w =0 , \quad \text{on} \quad {B(O,r)} \times (0, \infty);\\
w(\cdot, 0) =1,\ on\ B(O,r)\\
w(\cdot,t)=1,\ on\ \pa B(O,r).
\end{cases}
\ee
Since $V$ is a nonnegative smooth function, \eqref{eqforw} has a long time solution.

To show that $J$ exists for all time and derive the bounds for $J$, we derive the bounds for $w$ first. Via the Duhamel's formula, \eqref{eqforw} can be transformed to the following
integral equation,
\be
\label{inteqforw}
w(x, t) = 1+ 2(a-1)\int^t_0 \int_{B(O,r)} G_0(x, t-s; y, 0) V(y) w (y, s) dyds.
\ee Here $G_0(x,t;y,s)$ is the Dirichlet heat kernel on $B(O,r)$.\\

\hspace{-.45cm}{\it step 2. long time bounds}

For a lower bound of $w$, notice that since $a-1>0$, $V\geq 0$ and $w=1$ on the parabolic boundary of $B(O, r)\times [0,\infty)$, it follows from the maximum principle (see e.g. \cite{Lie} Lemma 2.1) that $w(x,t)\geq 1$ on $B(O,r)\times (0,\infty)$.

Next, we use Gronwall's inequality to derive an upper bound for $w$. Let $h(t)=\sup_{B(O,r)\times [0,t]}w(x,s)$. Note that $h(t)$ is nondecreasing since $w\geq 1$. Thus, it follows from \eqref{inteqforw} that
\be\label{eq h1}
h(t) \leq 1+ 2(a-1)\int^t_0 \int_{B(O,r)} G_0(x, t-s; y, 0) V(y) h(s) dyds.
\ee
To estimate the second term on the right, we consider the following two cases.

{\it Case 1:} If $t-s\geq r^2$, then by Lemma \ref{lem volume doubling}, we get
\be\label{vol1}
|B(z,\sqrt{t-s})|\geq|B(z,r)|\geq C(n)|B(z,2r)|\geq C(n)|B(O,r)|,
\ee
for any $z\in B(O,r)$.

Notice further that $ G_0(x, t-s; y, 0)\leq  G(x, t-s; y, 0)$, from \eqref{Gaussian upper bound}, \eqref{vol1} and H\"older inequality, we have, since $r \le 1$, that
\be\label{case1}
\al
\int_{B(O,r)}G_0(x,t-s;y,0)V(y)dy
\leq &\int_{B(O,r)}\frac{C_1}{|B(x,\sqrt{t-s})|^{\frac{1}{2}}|B(y,\sqrt{t-s})|^{\frac{1}{2}}}e^{-\frac{d^2(x,y)}{C_2(t-s)}}V(y)dy\\
\leq &\frac{C(n,p)}{|B(O,r)|}\int_{B(O,r)}V(y)dy\\
\leq &C(n,p)\left(\frac{1}{|B(O,r)|}\int_{B(O,r)}V(y)^p dy\right)^{1/p}\\
\leq& C(n,p)r^{-2}k(p,r).
\eal
\ee

{\it Case 2:} If $t-s\leq r^2$, then again by Lemma \eqref{lem volume doubling}, for any $z\in B(O, r)$ we have
\be
|B(z,\sqrt{t-s})|\geq 2\frac{(t-s)^{n/2}}{r^n}|B(z,r)|\geq C(n)\frac{(t-s)^{n/2}}{r^n}|B(O,r)|.
\ee
Thus, we have
\be\label{case2}
\al
\int_{B(O,r)}G_0(x,t-s;y,0)V(y)dy
\leq &||V||_{L^p,B(O,r)}\left(\int_{B(O,r)}G^{\frac{p}{p-1}}dy\right)^{\frac{p-1}{p}}\\
=& ||V||_{L^p,B(O,r)}\left(\int_{B(O,r)}G^{\frac{1}{p-1}}\cdot G dy\right)^{\frac{p-1}{p}}\\
\leq & ||V||_{L^p,B(O,r)} \frac{C_1^{1/p}}{|B(x,\sqrt{t-s})|^{\frac{1}{2p}}|B(y,\sqrt{t-s})|^{\frac{1}{2p}}}\\
\leq & C(n,p)||V||_{L^p,B(O,r)}\frac{r^{\frac{n}{p}}}{(t-s)^{\frac{n}{2p}}}\frac{1}{|B(O,r)|^{1/p}}\\
\leq& C(n,p)k(p,r)\frac{r^{\frac{n}{p}-2}}{(t-s)^{\frac{n}{2p}}}.
\eal
\ee
Here $||V||_{L^p,B(O,r)}=\left(\int_{B(O,r)}|V|^p \right)^{1/p}.$

Inserting \eqref{case1} and \eqref{case2} in \eqref{eq h1} yields
\be\label{eq h2}
\al
h(t)\leq & 1+2(a-1)\int_0^{t-r^2}+\int_{t-r^2}^t\left(\int_{B(O,r)}G(x,t-s;y,0)V(y)h(s)dy\right)ds\\
\leq & 1+C(n,p)(a-1)k(p,r)\left[r^{-2}\int_0^{t-r^2}h(s)ds+\int_{t-r^2}^t\frac{r^{\frac{n}{p}-2}}{(t-s)^{\frac{n}{2p}}}h(s)ds\right].
\eal
\ee
For the last term on the right above, we have
\[
\al
\int_{t-r^2}^t\frac{r^{\frac{n}{p}-2}}{(t-s)^{\frac{n}{2p}}}h(s)ds=&\int_{t-r^2}^{t-\e r^2}+\int_{t-\e r^2}^{t}\frac{r^{\frac{n}{p}-2}}{(t-s)^{\frac{n}{2p}}}h(s)ds\\
\leq & \e^{-\frac{n}{2p}}r^{-2}\int_0^t h(s)ds+\frac{2p}{2p-n}\e^{1-\frac{n}{2p}}h(t).
\eal
\]
Therefore, \eqref{eq h2} becomes
\[
\al
h(t)\leq & 1+C(n,p)(a-1)k(p,r)r^{-2}\int_0^{t}h(s)ds\\
&\ +C(n,p)(a-1)k(p,r)\left[\e^{-\frac{n}{2p}}r^{-2}\int_0^t h(s)ds+\frac{2p}{2p-n}\e^{1-\frac{n}{2p}}h(t)\right],
\eal
\]
i.e.,
\[
\left[1-C(n,p)(a-1)k(p,r)\e^{\frac{2p-n}{2p}}\right]h(t)\leq 1+C(n,p)(a-1)k(p,r)r^{-2}(1+\e^{-\frac{n}{2p}})\int_0^th(s)ds.
\]
By choosing $\e=\left[2C(n,p)(a-1)k(p,r)\right]^{-\frac{2p}{2p-n}}$ such that $$1-C(n,p)(a-1)k(p,r)\e^{\frac{2p-n}{2p}}=\frac{1}{2},$$ one gets
\[
h(t)\leq 2+2C(n,p)(a-1)k(p,r)r^{-2}\left(1+[2C(n,p)(a-1)k(p,r)]^{\frac{n}{2p-n}}\right)\int_0^th(s)ds,
\]
which is the Gr\"onwall inequality.

Thus, we obtain
\[
w(x,t)\leq h(t)\leq 2\exp\left\{2C(n,p)(a-1)k(p,r)r^{-2}\left(1+[2C(n,p)(a-1)k(p,r)]^{\frac{n}{2p-n}}\right)t\right\}.
\]
From Remark 2.2 in \cite{DWZ}, we know
\[
k(p,r)\leq 2^{1/p}k(p,1)\leq 2^{1/p}\kappa.
\]
It follows that
\[
w(x,t)\leq 2\exp\left\{2C(n,p)(a-1)\kappa r^{-2}\left(1+[2C(n,p)(a-1)\kappa]^{\frac{n}{2p-n}}\right)t\right\}.
\]
Since $w=J^{-(a-1)}$, we derive from above that
\[J\geq 2^{-\frac{1}{a-1}}\exp\left\{-2C(n,p)\kappa r^{-2}\left(1+[2C(a-1)\kappa]^{\frac{n}{2p-n}}\right)t\right\}.
\]
This finishes the proof of the claim.\\

Now let us continue the proof of the theorem. Let $J$ be the function in the claim with $r=1$, then \eqref{ineqtQ} becomes
\be\label{eq Q}
\al
(\Delta  - \pa_t) ( t Q )
+ 2 \frac{\d u}{u} \d (t Q)
&\ge  \a t \frac{2 - \delta }{n} J \left( | \d  f |^2 - \pa_t f \right)^2   - \delta \a  t J |\d f|^4 -Q.
\eal
\ee
According to Lemma \ref{cutoff}, we may choose a cut-off function $\phi$ satisfying
\be\label{eq cutoff}
0\leq \phi\leq 1,\quad supp\phi\subset\subset B(O,1),\quad \phi=1\ in\ B(O,\frac{1}{2}),\quad \ |\d \phi|^2+|\Delta \phi|\leq C(n,p).
\ee
In the following, we will use $C$ for constant $C(n,p)$ for simplicity. But the constants may be different from line to line.

From \eqref{eq Q} we have
\be\label{eq tphiQ}
\al
&t\phi^2(\Delta  - \pa_t) (  t\phi^2 Q )+ 2 t\phi^2\frac{\d u}{u} \d (t\phi^2 Q)\\
\ge&  \a t^2 \phi^4\frac{2 - \delta }{n} J  \left( | \d  f |^2 - \pa_t f \right)^2   - \delta \a  t^2 \phi^4 J |\d f|^4 -t\phi^4 Q+2t^2\phi^3Q\Delta \phi\\
& +2t^2\phi^2Q|\d \phi|^2+ 4t^2\phi^3\d Q\d\phi+4t^2\phi^3Q\frac{\d u}{u}\d \phi.
\eal
\ee
For any $T>0$, we may assume that $t\phi^2 Q$ achieves a positive maximum at some interior point $x\in B(O,r)$ and time $t\in(0,T]$, for otherwise, we have $Q\leq 0$ which is stronger than \eqref{Li-Yau}. Then, at $x$ and $t$ one has
\be\label{max point}
(\Delta -\pa_t)(t\phi^2Q)\leq 0,\ and\ \d(t\phi^2Q)=0,\ i.e.,\ \phi\d Q =-2Q\d \phi.
\ee
It follows from \eqref{eq cutoff}, \eqref{eq tphiQ} and \eqref{max point} that
\be\label{eq tphiQ2}
\al
0\geq& \a t^2 \phi^4\frac{2 - \delta }{n} J \left( | \d  f |^2 - \pa_t f \right)^2   - \delta \a  t^2 \phi^4 J |\d f|^4 -t\phi^4 Q+t^2\phi^3Q\Delta \phi\\
&\ - 6t^2\phi^2Q|\d\phi|^2+2t^2\phi^3Q\frac{\d u}{u}\d \phi\\
\geq & \a t^2 \phi^4\frac{2 - \delta }{n} J \left( | \d  f |^2 - \pa_t f \right)^2   - \delta \a  t^2 \phi^4 J |\d f|^4 -t\phi^2 Q-Ct^2\phi^2Q-Ct^2\phi^3Q|\d f|.
\eal
\ee
Notice that
\[
\al
\left( | \d  f |^2 - \pa_t f \right)^2=&\left[Q+(1-\a J)|\d f|^2\right]^2\\
=& Q^2+2(1-\a J)Q|\d f|^2+(1-\a J)^2|\d f|^4.
\eal
\]
One gets from \eqref{eq tphiQ2} that
\be\label{eq tphiQ3}
\al
0\geq &\a t^2 \phi^4\frac{2 - \delta}{n}JQ^2  + 2\a t^2 \phi^4\frac{2 - \delta}{n}J(1-\a J)Q|\d f|^2\\
&\ +\a  t^2 \phi^4 J\left[\frac{2 - \delta}{n}(1-\a J)^2- \delta\right] |\d f|^4 -t\phi^2 Q-Ct^2\phi^2Q-Ct^2\phi^3Q|\d f|.
\eal
\ee
By choosing
\be\lab{alphadelta}
\delta=\frac{2(1-\alpha)^2}{n+(1-\alpha)^2},
\ee one has
\be
\lab{conddelta}
\frac{2-\delta}{n} (1-\a )^2 - \delta  = 0.
\ee
Since $J\leq 1$, we derive from above that
$$\frac{2-\delta}{n} (1-\a J)^2 - \delta  \ge 0\quad \text{on} \quad B(O,1) \times [0, \infty).$$
Inserting this in \eqref{eq tphiQ3} induces
\[
\al
0\geq &\a t^2 \phi^4\frac{2 - \delta}{n}JQ^2  + 2\a t^2 \phi^4\frac{2 - \delta}{n}J(1-\a J)Q|\d f|^2-t\phi^2 Q-Ct^2\phi^2Q-Ct^2\phi^3Q|\d f|\\
=& \a\frac{2 - \delta}{n}\underline{J}(t\phi^2Q)^2+\left[2\a \frac{2 - \delta}{n}(1-\a)\underline{J}\phi^2|\d f|^2-C\phi|\d f|\right](t\phi^2 Q)t-t\phi^2 Q-Ct(t\phi^2Q)\\
\geq& \a\frac{2 - \delta}{n}\underline{J}(t\phi^2Q)^2-\frac{Ct}{\left[\a(2-\delta)(1-\a)\underline{J}\right]}(t\phi^2Q)-Ct(t\phi^2Q)-t\phi^2Q.
\eal
\]
It follows that
\[
t\phi^2Q\leq \frac{n}{\a(2 - \delta)\underline{J}}+\frac{Ct}{\a(2 - \delta)\underline{J}}\left[\frac{1}{\a(2-\delta)\underline{J}(1-\a)}+1\right],
\]
which implies that
\[
Q\leq \frac{n}{\a(2 - \delta)\underline{J}}\frac{1}{t}+\frac{C}{\a(2 - \delta)\underline{J}}\left[\frac{1}{\a(2-\delta)\underline{J}(1-\a)}+1\right]
\]
in $B(O,\frac{1}{2})\times(0,\infty)$.
\end{proof}

\begin{remark}\label{rmk}
The Li-Yau bound \eqref{Li-Yau} is a scaling invariant inequality. Indeed, under the parabolic scaling by a factor of $r^{2}$, i.e., let $\tilde{g}=r^{2}g$ and $\tilde{t}=r^{2}t$, the assumption $k(p,1)\leq \kappa$ reads $\tilde{k}(p,r)\leq \kappa$, and \eqref{Li-Yau} becomes
\[
\a \underline{\tilde{J}}_r\frac{|\d u|^2}{u^2}-\frac{\pa_{\tilde{t}} u}{u}\leq \frac{n}{\a(2 - \delta)\underline{\tilde{J}}_r}\frac{1}{\tilde{t}}+\frac{C}{\left(\a(2 - \delta)\underline{\tilde{J}}_r\right)r^2}\left[\frac{1}{\a(2-\delta)\underline{\tilde{J}}_r(1-\a \underline{\tilde{J}}_r)}+1\right],
\]
where \[
\underline{\tilde{J}}_r=\underline{\tilde{J}}_r(t)=2^{-\frac{1}{a-1}}\exp\left\{-2C\kappa r^{-2}\left(1+[2C(a-1)\kappa]^{\frac{n}{2p-n}}\right)\tilde{t}\right\}.
\]
\end{remark}

\hspace{-.5cm}{\bf Acknowledgements}
We are grateful to Professors H.-D. Cao, X.Z. Dai, H.Z. Li, G.F. Wei for their interest and comments on the result.

Q.S.Z. gratefully acknowledges the support of Simons'
Foundation


\begin{thebibliography}{99}

\bibitem[BBG]{BBG} Bakry, Dominique; Bolley, Francois; Gentil Ivan, {\it The Li-Yau inequality and applications under a curvature-dimension condition},  arXiv:1412.5165, 2014.

\bibitem[BL]{BL} Bakry, Dominique; Ledoux, Michel,
{\it A logarithmic Sobolev form of the Li-Yau parabolic inequality},
  Rev. Mat. Iberoamericana, Volume 22, Number 2 (2006), 683-702.

\bibitem[CaNi]{CaNi} Cao, Huai-Dong; Ni, Lei, {\it Matrix Li-Yau-Hamilton estimates for the heat equation on K\"ahler manifolds}, Math. Ann. 331 (2005), no. 4, 795-807.

\bibitem[CTZ]{CTZ} Cao, Huai-Dong; Tian, Gang; Zhu, Xiaohua, {\it K\"ahler-Ricci solitons on compact complex manifolds with $C_1(M)>0$}. Geom. Funct. Anal. 15 (2005), no. 3, 697-719.

\bibitem[DWZ]{DWZ} Dai, Xianzhe; Wei, Guofang; Zhang, Zhenlei, {\it Local Sobolev Constant Estimate for Integral Ricci Curvature Bounds}, arXiv:1601.08191.

\bibitem[Dav]{Dav} Davies, E. B., {\it Heat  kernels  and  spectral  theory}, volume  92  of
Cambridge  Tracts  in Mathematics,   Cambridge University Press, Cambridge, 1989

\bibitem[GM]{GM} Garofalo, N.; Mondino, A.,  {\it  Li-Yau and Harnack type inequalities in
$RCD^*(K;N)$ metric measure spaces.} Nonlinear Anal., 95, 721-734,  2014.

\bibitem[Ha]{Ha} Hamilton, Richard S., {\it A matrix Harnack estimate for the heat equation}, Comm. Anal. Geom. 1 (1993), no. 1, 113-126.

\bibitem[LX]{LX} Li, J. F.; Xu, X.J.,  {\it Differential Harnack inequalities on Riemannian manifolds I: linear heat equation.} Adv. Math., 226(5), 4456-4491, 2011.

\bibitem[LY]{LY} Li, Peter; Yau, Shing-Tung, {\it On the parabolic kernel of the Schr\"odinger operator.}  Acta Math. 156 (1986), no. 3-4, 153-201.

\bibitem[Lie]{Lie}  Lieberman, Gary M., {\it Second order parabolic differential equations.} World Scientific Publishing Co., Inc., River Edge, NJ, 1996.

\bibitem[PeWe]{PeWe} Petersen, Peter; Wei, Guofang, {\it Relative volume comparison with integral curvature bounds}, Geom. Funct. Anal. 7 (1997), no. 6, 1031-1045.

\bibitem[QZZ]{QZZ} Qian, Z.; Zhang, H.-C.; Zhu X.P.,  {\it Sharp spectral gap and Li-Yau's estimate on
Alexandrov spaces}. Math. Z., 273(3-4), 1175-1195, 2013.

\bibitem[Sa]{Sa} Saloff-Coste, Laurent, {\it Uniformly elliptic operators on Riemannian manifolds}, J. Differential Geom. 36 (1992), no. 2, 417-450.

\bibitem[Wan]{Wan}  Wang, F.-Y., {\it Gradient and Harnack inequalities on noncompact manifolds with boundary}, Pacific J. Math., 245(1), 185-200, 2010.

\bibitem[WanJ]{WanJ} Wang, Jiaping, {\it Global heat kernel estimates.} Pacific J. Math. 178 (1997), no. 2, 377-398.

\bibitem[ZhZx]{ZhZx} Zhang, Hui-Chun; Zhu, Xi-Ping, {\it Local Li-Yau's estimates on $RCD^*(K,N)$ metric measure spaces}, arXiv:1602.05347, to appear in Cal. Var. PDE.

\bibitem[ZZ]{ZZ} Zhang, Qi S.; Zhu, Meng, {\it Li-Yau gradient bounds under nearly optimal curvature conditions},  arXiv:1511.00791.
\end{thebibliography}
\end{document}